\documentclass[letterpaper, 10 pt, conference]{ieeeconf}  

\IEEEoverridecommandlockouts                              
\usepackage{graphics} 
\usepackage{epsfig} 
\usepackage{amsmath} 

\usepackage{amssymb}
\usepackage{amsthm}  
\usepackage{amscd}
\usepackage{dsfont}
\usepackage{cite}
\usepackage{float}
\usepackage{times}
\usepackage{xcolor}
\usepackage{comment}
\newtheorem{theorem}{Theorem}
\newtheorem{remark}{Remark}
\newtheorem{lemma}{Lemma}

\newtheorem{proposition}{Proposition}

\newtheorem{definition}{Definition}
\newtheorem{assumption}{Assumption}

\newcommand{\beq}{\begin{equation}}
\newcommand{\eeq}{\end{equation}}
\newcommand{\beqa}{\begin{eqnarray}}
\newcommand{\eeqa}{\end{eqnarray}}
\newcommand{\beqan}{\begin{eqnarray*}}
\newcommand{\eeqan}{\end{eqnarray*}}

\newcommand{\Rset}{\mathbb{R}}

\newcommand{\Acal}{{\cal A}}

\newcommand{\Ccal}{{\cal C}}

\newcommand{\Gcal}{{\cal G}}

\newcommand{\Lcal}{{\cal L}}

\newcommand{\Wcal}{{\cal W}}
\newcommand{\Xcal}{{\cal X}}

\renewcommand{\(}{\left(}
\renewcommand{\)}{\right)}

\newcounter{l1}
\newcounter{l2}
\newcounter{l3}
\newcommand{\bdotlist}{\begin{list}{$\bullet$}{}}
\newcommand{\bboxlist}{\begin{list}{$\Box$}{}}
\newcommand{\bbboxlist}{\begin{list}{\raisebox{.005in}{{\tiny
$\blacksquare$ \ \ }}}{}}
\newcommand{\bdashlist}{\begin{list}{$-$}{} }
\newcommand{\blist}{\begin{list}{}{} }
\newcommand{\barablist}{\begin{list}{\arabic{l1}}{\usecounter{l1}}}
\newcommand{\balphlist}{\begin{list}{(\alph{l2})}{\usecounter{l2}}}
\newcommand{\bAlphlist}{\begin{list}{\Alph{l2}.}{\usecounter{l2}}}
\newcommand{\bdiamlist}{\begin{list}{$\diamond$}{}}
\newcommand{\bromalist}{\begin{list}{(\roman{l3})}{\usecounter{l3}}}

\title{\LARGE \bf
On Efficient Aggregation of Distributed Energy Resources
}

\author{Zuguang Gao, Khaled Alshehri, and John R. Birge
\thanks{Z. Gao and J. R. Birge are with Booth School of Business, The University of Chicago, Chicago, IL 60637, USA. 
        {\tt\small \{zuguang.gao, john.birge\}@chicagobooth.edu}}%
\thanks{K. Alshehri is with the Systems Engineering Department, King Fahd University of Petroleum and Minerals (KFUPM), Dhahran, Saudi Arabia. 
        {\tt\small kalshehri@kfupm.edu.sa}}      
}

\begin{document}

\maketitle
\thispagestyle{empty}
\pagestyle{empty}

\begin{abstract}
	
The rapid expansion of distributed energy resources (DERs) is one of the most significant changes to electricity systems around the world. Examples of DERs include solar panels, small natural gas-fueled generators, combined heat and power plants, etc. Due to the small supply capacities of these DERs, it is impractical for them to participate directly in the wholesale electricity market. We study in this paper an efficient aggregation model where a profit-maximizing aggregator procures electricity from DERs, and sells them in the wholesale market. The interaction between the aggregator and the DER owners is modeled as a Stackelberg game: the aggregator adopts two-part pricing by announcing a participation fee and a per-unit price of procurement for each DER owner, and the DER owner responds by choosing her payoff-maximizing energy supplies. We show that our proposed model preserves full market efficiency, i.e., the social welfare achieved by the aggregation model is the same as that when DERs participate directly in the wholesale market. We also note that two-part pricing is critical for market efficiency, and illustrate via an example that with one-part pricing, there will be an efficiency loss from DER aggregation, due to the profit-seeking behavior of the aggregator.


\end{abstract}

\section{INTRODUCTION}

The increasing adoption of distributed energy resources (DERs) challenges fundamental assumptions in existing electricity markets design and operation, as end-consumers not only become proactive, but also small-scale producers who can have significant impacts on the power system~\cite{ritzenhofen2016structural}. According to the North American Electric Reliability Corporation
(NERC),  DER is ``any resource on the distribution system that produces electricity and is not otherwise included in the formal NERC definition of the Bulk Electric System (BES)"\cite{derdefinition}. The existence of DERs causes a fundamental shift, because electric power demand is largely assumed to be fixed from independent market operators' (ISOs) perspective \cite{eugene,eugene2,eugene3,EnergyReserve,feng}. This fundamental shift calls for the quantification of the impacts of DERs on wholesale markets design. Such quantification is not easy; ISOs do not have oversight over the distribution power network, and hence, cannot include DER owners as market participants. Furthermore, even if ISOs can oversee the distribution power system, it would be a significant burden and impractical to include DER supply directly into the wholesale electricity market operations through ISOs, due to the communicational, computational, and operational complexity. \footnote{The number of market participants would grow from hundreds to millions.}

\begin{figure}
	\centering
	\includegraphics[width=.8\linewidth]{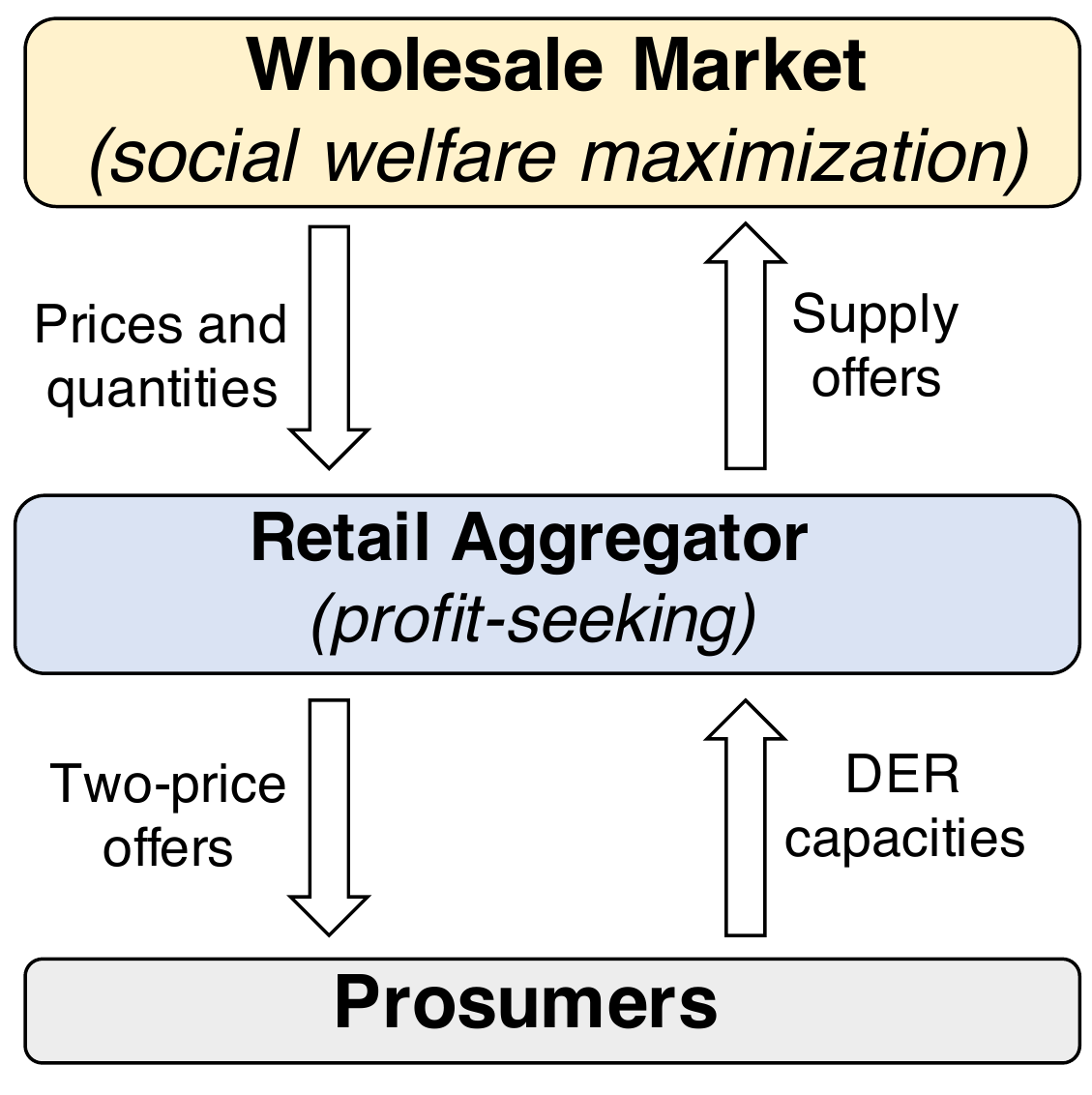}
	\caption{Overall interactions in the proposed efficient aggregation model}
	\label{sketch}
\end{figure}

Different models have been proposed to include DER supply into wholesale electricity markets. One possibility is to have a Distribution System Operator (DSO) acting as a market manager at the distribution level, and finding socially-optimal dispatch, similar to ISOs \cite{Lian,Ntakou,Manshadi,Sotkiewicz,Terra,Huang}. Another possibility is to have fully-distributed electricity market designs, in which end-consumers can trade energy among themselves \cite{Moret,Rahimi}. The third model, which we adopt here, is to have an aggregating company whose role is to collect energy from DER owners, and participate in the wholesale market as a producer of electricity. This model has been adopted by California ISO and New York ISO~\cite{CAISODER,DER_NYISO},  and seems to be realistic for practical implementation, especially with FERC's recent Order No. 2222 in September, 2020~\cite{FERC2222}. The aggregator here buys DER supply from their owners, and bids directly to the ISO similar to generating companies. The relationship between the aggregator and prosumers in the same geographical footprint is naturally monopolistic, where aggregators become price-making in retail electricity markets as they can send price offers to prosumers in order to collect DER supply. Such price offers need to be high-enough so that DER owners are attracted to sell, but small-enough so that the aggregator can maximize her profits. This profit-seeking behavior can impact the overall electricity market efficiency, but at the same time, due to the impracticality of direct DER participation into wholesale markets, aggregators are necessary and important players. This gives rise to the following important question: {\em In the presence of a profit-seeking and a monopolistic aggregator, is there an aggregation model that can attain a socially optimal (efficient) market outcome?} The presence of such a mechanism can in fact be significant. First, in reality, aggregators are mostly profit-seeking, and are often monopolistic, which makes the markets prone to efficiency losses as demonstrated in \cite{CISS,TPS}. Second, it is infeasible for DER owners to participate in wholesale markets, so the presence of such intermediaries is inevitable. Third, if such a mechanism can be designed, it would address various debates surrounding whether or not DER aggregation need to be done by profit-making entities, or social intermediaries. 

In this paper, we address the above question by proposing a DER aggregation model that yields efficient market outcomes, even when a DER aggregator $\Acal$ is monopolistic and profit-seeking. We analytically show that with our proposed model, the optimal DER capacities being integrated through $\Acal$ are equivalent to a benchmark (ideal, not realistic) case in which DER owners can directly participate in the wholesale market. Briefly, our aggregation model utilizes a two-price offers $\bf(P,p)$ from $\Acal$ to the prosumers, where $\bf P$ corresponds to a fixed DER owner participation costs (connection charges), and the other prices $\bf p$ are for marginal acquisition of DER capacities. We remark here that having one-price offer $\bf p$ would not yield efficient market outcomes due to the monopolistic behavior of $\Acal$, as shown in \cite{CISS,TPS}. For an overall illustration, refer to Figure \ref{sketch}. Here, we adopt a game-theoretic analysis of the relationship between the prosumers and $\Acal$, and show that the equilibrium prices  $(\bf P^*,p^*)$ and DER capacities $\bf{x}^*(\bf P^*,p^*)$ attain efficient market outcomes. 

The paper is organized as follows: Section II discusses the benchmark ideal model. In Section III, we propose an efficient DER integration model, and state the main theoretical results of the paper. Via illustrative examples, we conduct simulations in Section IV to gain some practical insights. The paper concludes in Section V. 

\subsubsection*{Notation}
\label{sec:notation}
We let $\Rset$ denote the set of real numbers, and $\Rset_+$ (resp. $\Rset_{++}$) denote the set of nonnegative (resp. positive) numbers. For $z\in\Rset$, we let $[z]^+ := \max\{z, 0\}$. 
In any optimization problem, a decision variable $x$ at optimality is denoted by $x^*$. We use boldface to distinguish a vector $\bf{x}$ from a scalar $x$. For an event $\Xcal$, we let $\mathds{1}\{\Xcal\}$ denote the indicator function.   

\section{DIRECT PROSUMER PARTICIPATION MODEL (BENCHMARK)}

In this section, we introduce an ideal model where prosumers can participate (buy or sell) directly in the wholesale market, and no aggregator is present. This model, though not realistic, serves as a benchmark for evaluating the market efficiency of our following models. There are three parties in this model:
prosumers who can buy and sell energy directly in the wholesale market, conventional generators who generate and sell electricity in the wholesale market, and an independent system operator (ISO) who clears the wholesale market. In the following, we describe the optimization problems solved by each of these three parties. Figure \ref{sketch2} illustrates the interactions between prosumers and the ISO.  


\begin{figure}
	\centering
	\includegraphics[width=.8\linewidth]{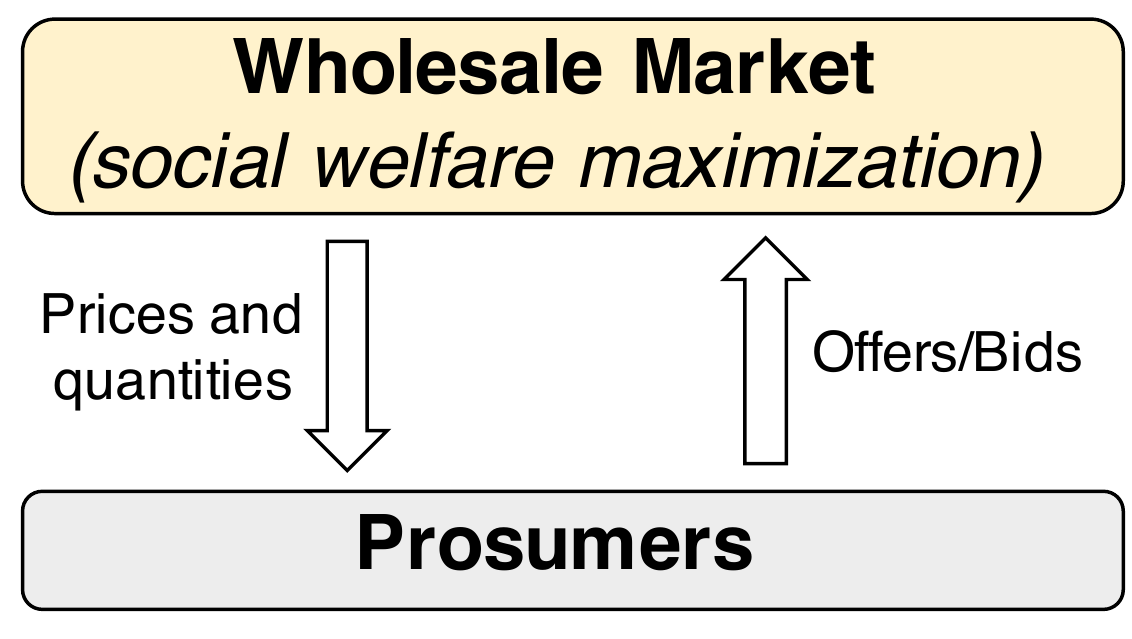}
	\caption{Benchmark model}
	\label{sketch2}
\end{figure}

\subsection{Prosumer's Problem}
Consider a power network with $n$ nodes (locations). At each location~$i$, there are~$N_i$ number of prosumers. We assume that each prosumer~$j$ at location~$i$ can purchase energy at the wholesale price~$\lambda^i$. Furthermore, prosumer~$j$ at location~$i$ is endowed with a capacity~${C}^i_j\ge 0$ of power production\footnote{$C^i_j$ can be zero for consumers with no DER supply, so the model includes those regular consumers.} from a collection of resources, such as solar panels, wall-mounted batteries, and plug-in electric vehicles. The power produced can be consumed locally by prosumer~$j$ or sold (partially) back to the wholesale market, again at the wholesale price~$\lambda^i$. The prosumer observes~$\lambda^i$ and decides the amount of energy to buy/sell. Let $u^i_j$ be prosumer~$j$'s utility of power consumption. We make the following assumption on prosumers' utility of consumption.
\begin{assumption}\label{assum:prosu}
	Each prosumer's utility of consumption~$u^i_j$ is increasing, strictly concave, and differentiable.
	Furthermore, the domain of~$u^i_j$ is $[0,Z]$ where $Z > C^i_j$ is some (large) upper bound of the amount of energy a prosumer can consume. We assume that $\frac{\partial u^i_j(z)}{\partial z}\to \infty$ as $z\to 0$, and $\frac{\partial u^i_j(z)}{\partial z}\to 0 $ as $z\to Z$.
\end{assumption}

We can then write prosumer~$j$'s optimization problem as
\begin{equation}\label{direct:pros}
	\begin{aligned}
		&\max_{x^i_j-d^i_j} \pi^i_j(x^i_j,d^i_j):=\lambda^i(x^i_j-d^i_j)+u^i_j(C^i_j+d^i_j-x^i_j)\\
		&\quad \text{s.t. }\qquad  C^i_j-Z\le x^i_j-d^i_j \le C^i_j,
	\end{aligned}
\end{equation}
where $d^i_j$ is the amount of energy prosumer~$j$ purchases and $x^i_j$ is the amount of energy she sells. For selling~$x^i_j$ at the wholesale price, prosumer~$j$ receives~$\lambda^ix^i_j$; for buying~$d^i_j$ at the wholesale price, she is charged~$\lambda^id^i_j$. The prosumer leaves with~$C^i_j +d^i_j-x^i_j$ to consume and her utility from consumption would be $u^i_j(C^i_j +d^i_j-x^i_j)$. While Assumption~\ref{direct:pros} further imposes strict concavity of prosumers' utilities, our analysis throughout this paper remains largely applicable to generic concave utilities, but strict concavity allows us to derive unique analytical solutions and gain deep insights. 

 We note that since prosumers at location~$i$ can buy and sell energy at the same price~$\lambda^i$, prosumer~$j$ essentially has $x^i_j-d^i_j$ as the single decision to make in solving~\eqref{direct:pros}, while we keep both $x^i_j$ and $d^i_j$ in the notation to clearly represent the amount of selling and buying, and to be consistent with the notation in the following sections. 
\begin{lemma} \label{lem:prosdirect} 
	Under Assumption~\ref{assum:prosu}, given any~$\lambda^i$, there exists a unique optimal solution $(x^i_j-d^i_j)^*$ for the prosumer's problem~\eqref{direct:pros} which satisfies
	\begin{align}\label{eq:prosopt}
		\frac{\partial u^i_j(z)}{\partial z}\bigg|_{z=C^i_j-(x^i_j-d^i_j)^*} = \lambda^i.
	\end{align}
\end{lemma}

The above Lemma directly follows from the properties of problem \eqref{direct:pros}. Without loss of generality, and for ease of exposition, we restrict our attention to the case in which at most one of $x^i_j$ and $d^i_j$ can be nonzero.
We first solve~\eqref{direct:pros} for $(x^i_j-d^i_j)^*$. Then, we let ${x^i_j}^* = \left[(x^i_j-d^i_j)^*\right]^+$ and ${d^i_j}^* = \left[-(x^i_j-d^i_j)^*\right]^+$. We also use the notation~${x^i_j}^*(\lambda^i)$ and~${d^i_j}^*(\lambda^i)$ to denote the optimal response of prosumer~$j$ at location~$i$ for a given wholesale market price~$\lambda^i$.

\subsection{Generator's Problem}\label{sec:benchgene}
Let $G_i$ denote the number of dispatchable conventional generators at location~$i$. Generator~$j$ at location~$i$ chooses to supply $y^i_j \in \left[\underline{y}^i_j,\overline{y}^i_j\right] $. Let $c^i_j(y^i_j)$ be denote the production cost. We have the following assumption. 
\begin{assumption}\label{assum:gene}
	Each generator's cost function~$c^i_j$ is increasing, strictly convex, and differentiable in $\left[\underline{y}^i_j,\overline{y}^i_j\right]$. Furthermore, we let $\frac{\partial c^i_j(y^i_j)}{\partial y^i_j}\to 0$ as $y^i_j\to \underline{y}^i_j$ and $\frac{\partial c^i_j(y^i_j)}{\partial y^i_j}\to \infty$ as $y^i_j\to \overline{y}^i_j$.
\end{assumption}
By selling $y^i_j$, generator~$j$ earns a compensation $\lambda^iy^i$. Given a wholesale price $\lambda^i$, generator~$j$ maximizes its payoff by solving
\begin{align}\label{direct:gene}
	\max_{y^i_j\in\left[\underline{y}^i_j,\overline{y}^i_j\right]}\hat{\pi}^i_j(y^i_j) := \lambda^i_jy^i_j-c^i_j(y^i_j).
\end{align}
\begin{lemma}\label{lem:geneopt}
	Under Assumption~\ref{assum:gene}, given any~$\lambda^i$, there exists a unique optimal solution~${y^i_j}^*$ for the generator's problem~\eqref{direct:gene} which satisfies
	\begin{align}\label{eq:geneopt}
		\frac{\partial c^i_j(y^i_j)}{\partial y^i_j}\bigg|_{y^i_j = {y^i_j}^*} = \lambda^i.
	\end{align}
\end{lemma}
The above Lemma directly follows from the properties of problem \eqref{direct:gene}. We also use the notation~${y^i_j}^*(\lambda^i)$ to denote the optimal response of generator~$j$ at location~$i$ at a given wholesale price~$\lambda^i$.

\subsection{The Economic Dispatch Problem}\label{sec:benchiso}

Wholesale electricity markets in the United States and other countries are managed by independent system operators (ISOs).  An ISO clears the market by matching supply and demand via social welfare maximization (in practice, this is often done by production cost minimization to meet fixed system demands \cite{eugene2,eugene3,EnergyReserve}), while ensuring that the power flows satisfy the network and line capacity constraints. Specifically, let $X^i = \sum_{j\in[N_i]}x^i_j$ be the total power supply at node~$i$ from prosumers; let $Y^i = \sum_{j\in[G_i]}y^i_j$ be the total power supply at node~$i$ from conventional generators; and let $D^i=\sum_{j\in [N_i]}d^i_j$ be the total demand at node~$i$. Furthermore, we let ${\bf f}$ be the vector of capacities of transmission lines in the power network, and let ${\bf B}$ be the line-to-node incidence matrix. The system operator chooses a vector ${\bf h}$, where each element $h^i$ is the net injection to node~$i$. We have the following network constraints:
\begin{subequations}\label{eq:directcons}
	\begin{align}
		&{\bf h = D - Y - X},\label{eq:demsup}\\
		&{\bf 1}^T {\bf h}= 0, \qquad {\bf Bh}\le {\bf f}. \label{eq:netcons}	\end{align}
\end{subequations}
We note that~\eqref{eq:demsup} ensures the total supply matches the total demand at each node;~\eqref{eq:netcons} ensures that the total net injection by the system operator is zero over the power network (here, ${\bf 1}$ is a vector of ones), and the total power transmission at each line does not exceed its capacity. In addition to the network constraints, the ISO needs to also consider all participant-specific constraints described earlier in problems \eqref{direct:pros} and \eqref{direct:gene}: 
\begin{equation}\label{eq:directcons_p}
		{\bf C-Z\le x-d \le C, \qquad  \underline{y} \le y  \le \overline{y} } 	
\end{equation}
The objective of the system operator is to maximize the social welfare, which includes the prosumer surplus~({\sf PS}), generator surplus~({\sf GS}), and merchandizing surplus ({\sf MS}):
\begin{subequations}
	\begin{align}
		{\sf PS} &:= \sum_{i\in[n]}\left(\sum_{j\in [N_i]}u^i_j(d^i_j - x^i_j + {C}^i_j) - \lambda^i(h^i + Y^i)\right),\label{directcs}\\
		{\sf GS} &:= \sum_{i\in[n]}\sum_{j\in [G_i]}\left(\lambda^iy^i_j - c^i_j(y^i_j)\right), \ {\sf MS} := \sum_{i\in[n]}\lambda^ih^i\label{directms}.
	\end{align}
\end{subequations}
where we have imposed the relationship~\eqref{eq:demsup} in deriving~\eqref{directcs}. The social welfare that the system operator optimizes is the sum of {\sf PS}, {\sf GS}, and {\sf MS}. After canceling terms, the social welfare can be written as
\begin{align}\label{eq:WB}
	\mathcal{W}_B:= \sum_{i\in[n]} \left(\sum_{j\in[N_i]}u^i_j(d^i_j +{C}^i_j-x^i_j) -  \sum_{j\in [G_i]}c^i_j(y^i_j)\right).
\end{align}
The system operator's economic dispatch problem is then:
\begin{equation}\label{eq:so}
	\begin{aligned}
		&\max \qquad \mathcal{W}_B ({\bf h, x-d, y})\\
		&  \text{subject to}\qquad \eqref{eq:directcons} - \eqref{eq:directcons_p}.
	\end{aligned}
\end{equation}
\begin{assumption}\label{assum:exist}
	The system operator's economic dispatch problem~\eqref{eq:so} is feasible.
\end{assumption}

%
\begin{proposition}[Competitive Equilibrium]\label{prop:compeq}
	Under Assumptions~\ref{assum:prosu}-\ref{assum:exist}, there exists a unique optimal solution $\({\bf h^*, (x-d)^*, y^*}\)$ to~\eqref{eq:so}. Denote the optimal Lagrange multipliers of constraints~\eqref{eq:demsup} by ${\pmb{\lambda}}$. Then, the following statements are true:
	\begin{itemize}
		\item $\left(\bf x-d\right)^*$ and $\pmb{\lambda}$ satisfy~\eqref{eq:prosopt}. 
				\item ${\bf y}^*$  and $\pmb{\lambda}$ satisfy~\eqref{eq:geneopt}.
	\end{itemize}
\end{proposition}
The above proposition states that solving the system operator's problem~\eqref{eq:so} leads to a competitive equilibrium. From the prospective of prosumer $j$ at node $i$, this means that given the wholesale market price~$\lambda^i$ (the optimal Lagrange multiplier of ~\eqref{eq:demsup} for the same node), the corresponding solution to her problem, which satisfies~\eqref{eq:prosopt}, is the same as the optimal decision made by the ISO via solving ~\eqref{eq:so}. This is also true for all other prosumers and generators. Having all market participants being satisfied with the the competitive equilibrium, it serves as a good benchamrk for market efficiency.


\section{PROPOSED EFFICIENT AGGREGATION MODEL}

The direct participation model introduced in the previous section is a benchmark: prosumers are allowed to participate directly and sell their production in the wholesale market. In reality, the supply capacities of prosumers are typically too small for consideration in the wholesale market. \footnote{California ISO requires a minimum of $0.5 {\sf MW}$ for a DER aggregator to participate \cite{CAISO}, and New York ISO requires $0.1 {\sf MW}$.} Also, computing the dispatch and settlement for a large number of prosumers raises an untenable computational burden on the system operator. The presence of DER aggregators brings benefits to the system, as they open the door for DER owners to participate and bring more flexibility to the grid. However, the profit-seeking nature of these aggregators can cause efficiency losses \cite{CISS,TPS}. To resolve this, we propose in this section an efficient aggregation model under two-part pricing. In this model, prosumers sell part of their DER supply productions to an aggregator~$\mathcal{A}$, based on the price offers made by $\mathcal{A}$. The interactions between~$\mathcal{A}$ and prosumers are modeled as a Stackelberg game \cite{basar}. The aggregator acts as a leader and announces a price pair~$(P^i_j,p^i_j)$ for each prosumer~$j$ at location~$i$. The prosumer follows by choosing the amount of energy to sell. If the prosumer decides to sell a nonzero fraction of her capacity to the aggregator, she pays to the aggregator a participation fee~$P^i_j$, and earns the price~$p^i_j$ for each unit of energy sold. The aggregator~$\mathcal{A}$ then sells all these procured capacity to the wholesale market at the wholesale price~$\lambda^i$. The goal of the aggregator $\mathcal{A}$ is to choose prices~$(P^i_j,p^i_j)$ that maximize her profit, while anticipating how DER owners would respond. Note that here, DER owners have access only to one aggregator, so $\Acal$ is in fact monopolistic, which further signifies the importance of our mechanism as it yields socially-optimal outcomes.  

We also note that differential pricing is allowed in this model, i.e., the price pair~$(P^i_j, p^i_j)$ can be set differently for different prosumers. While it is reasonable to have varying prices depending on locations~\cite{birge2017inverse}, there exist some debates on whether prices should be allowed to set differently for prosumers from the same location. The legal issue is not the main focus of this paper. Though we assumed differential pricing in the model, the equilibrium we obtain in the end has the same marginal price at each location, that is, $p^i_j = \lambda^i,\ \forall j\in [N_i]$. The participation fee can be differentiated by, for example, mailing different coupons to different prosumers to encourage their participations, which is arguably more justifiable.
In the remainder of this section, we show that this aggregation model achieves the same socially-optimal market outcomes as in the direct participation model. 
\subsection{Prosumer's Problem}
Consider prosumer~$j$ at location~$i$. Upon seeing the prices~$(P^i_j,p^i_j)$ announced by~$\mathcal{A}$, prosumer~$j$ decides if she would sell part of her capacity to~$\mathcal{A}$. If she chooses so, she would pay a fee~$P^i_j$ to~$\mathcal{A}$, and receives $p^i_j$ for each unit of energy sold. We may write prosumer~$j$'s payoff as
\begin{align}\label{eq:aggpros}
	&\pi_{j}^i(x^i_j, d^i_j):=\nonumber\\
	&\qquad \begin{cases}
		p^i_jx^i_j - P^i_j + u^i_j\left(d^i_j+{C}^i_j-x^i_j\right) - \lambda^i d^i_j, \text{ if } x^i_j>0,\\
		u^i_j\left(d^i_j + {C}^i_j\right) - \lambda^id^i_j,
		 \hfill \text{ if }x^i_j=0,
	\end{cases} 
\end{align}
Given $(P^i_j, p^i_j,\lambda^i)$, prosumer~$j$ solves:
\begin{align}
	\max_{x^i_j\in [0, C^i_j], d^i_j\in[0,Z-C^i_j+x^i_j]}  \ \pi_{j}^i(x^i_j, d^i_j),
\end{align}
where $d^i_j$ is the amount of energy prosumer~$j$ purchases at wholesale market price $\lambda^i$, \footnote{In real markets, end-use customers are represented in the wholesale market via utility companies. Here, for ease of exposition, we omit the analysis of the relationship between utility companies and end-use customers and assume that they represent the true preferences of their customers in the wholesale market without modifications. This allows us to focus on the strategic aggregation component of the problem. Nevertheless, our insights remain largely applicable if this assumption is relaxed.}  and $x^i_j$ is the amount of energy she sells to the aggregator. For buying~$d^i_j$ at the wholesale price, prosumer~$j$ is charged~$\lambda^id^i_j$. If the prosumer does not sell ($x^i_j = 0$), she has a total of $d^i_j + C^i_j$ to consume, and her utility from consumption would be~$u^i_j(d^i_j + C^i_j)$. If the prosumer chooses to sell $x^i_j>0$ to~$\mathcal{A}$, she is charged a participation fee~$P^i_j$, and receives a compensation~$p^i_jx^i_j$. The prosumer would have~$d^i_j + C^i_j - x^i_j$ to consume in this case and her utility from consumption would be $u^i_j(d^i_j + C^i_j - x^i_j)$. 

Let ${x^i_j}^*(P^i_j,p^i_j,\lambda^i)$ and ${d^i_j}^*(P^i_j,p^i_j,\lambda^i)$ denote the optimal response of prosumer~$j$ given aggregator's announced prices~$(P^i_j,p^i_j)$ and the wholesale market price~$\lambda^i$.\footnote{Sometimes, we drop arguments for simplicity.} Note that if $p^i_j > \lambda^i$, the prosumer can arbitrage by buying at the price~$\lambda^i$ and selling at a higher price~$p^i_j$. This will result in the prosumer earning infinite payoff and aggregator losing infinite profit, which would be avoided by the aggregator. Therefore, we may without loss of generality restrict our discussions to the case when $p^i_j\le \lambda^i, \forall i\in[n], j\in[N_i]$. In the case~$p^i_j = \lambda^i$, similar to the direct participation model, we may enforce that~$x^i_j$ and $d^i_j$ cannot both be nonzero. We then have the following lemma on the optimal response of prosumers.

\begin{lemma}\label{lem:prosopt}
	Consider an arbitrary prosumer~$j$ at location~$i$. Let $(z_1,z_2)$ be such that 
	\begin{align}
		\frac{\partial u^i_j(z)}{\partial z}\Big|_{z=z_1} = \lambda^i, \qquad \frac{\partial u^i_j(z)}{\partial z}\Big|_{z=z_2} = p^i_j.\label{eq:z}
	\end{align}
	Then, under Assumption~\ref{assum:prosu}, both $(z_1,z_2)$ exist and are unique. Furthermore, prosumer~$j$'s optimal response can be described as follows.
	\begin{itemize}
		\item If $C^i_j\le z_2$, then, we have  $${d^i_j}^* = \left[z_1 - C^i_j\right]^+, \qquad {x^i_j}^*= 0.$$
		\item If $C^i_j > z_2$, then, we have 
		$${x^i_j}^*= \left(C^i_j-z_2\right)\cdot \mathds{1}\left\{\Xcal\right\}, \qquad {d^i_j}^* = 0,$$
		where $$\Xcal:=\{P^i_j\le p^i_j\left(C^i_j-z_2\right)+ u^i_j\left(z_2\right) - u^i_j\left(C^i_j\right)\}.$$
	\end{itemize}
\end{lemma}

The above Lemma states that if the capacity exceeds a certain value $z_2$ (note that $z_2$ is the value at which the marginal utility of consumption is equal to the aggregator's marginal price offer $p^i_j$)  and the upfront fee $P^i_j$ is not too high, then prosumers have an incentive to sell DER supply to the aggregator. If the capacity is too small or the upfront fee is too high, then prosumers would not sell DER supply and prefer to consume it locally.  
\subsection{Aggregator's Problem}
The DER aggregator~$\mathcal{A}$ collects power from prosumers and sells it in the wholesale market. By offering the prices~$(P^i_j,p^i_j)$ to each prosumer~$j$ at location~$i$, $\mathcal{A}$ procures a total capacity of $\sum_{j\in[N_i]}{x^i_j}^*(P^i_j,p^i_j)$ from location~$i$. $\mathcal{A}$ then sells it at the wholesale market price~$\lambda^i$. Given the wholesale market price~$\lambda^i$, the aggregator's profit from prosumer~$j$ at location~$i$ is
\begin{equation}
	\begin{aligned}
		&\Pi^i_j\left(P^i_j, p^i_j\right)
		:=\\&\qquad  P^i_j\mathds{1}\left\{{x^i_j}^*(P^i_j,p^i_j) > 0\right\}
		+ (\lambda^i-p^i_j){x^i_j}^*(P^i_j,p^i_j).
	\end{aligned} 
\end{equation}
Anticipating the response functions ${x^i_j}^*(P^i_j,p^i_j)$'s, she seeks to maximize her overall profit:
	\begin{align}\label{eq:aggobj}
		\max_{\bf P \ge 0, p \ge 0}   \underbrace{\sum_{i\in[n]}\sum_{j\in[N_i]} \Pi^i_j\left(P^i_j, p^i_j\right)}_{=:\hat{\Pi}(\bf P,p)},
	\end{align}
Aggregator $\mathcal{A}$'s profit is composed of two parts; the total participation fees charged to prosumers who sell positive amount of energy and the profits earned by reselling the procured energy in the wholesale market. We note that $\Acal$'s profit from prosumer~$j$ depends only on the response and the prices of prosumer~$j$, and there is no coupling among prosumers' problems. It should be clear that the Aggregator's problem~\eqref{eq:aggobj} can be decomposed to optimizing~$(P^i_j,p^i_j)$ for each prosumer at each location:
\begin{align}\label{eq:aggobj1}
	\max_{P^i_j\ge 0,p^i_j\ge 0}\Pi^i_j\left(P^i_j, p^i_j\right).
\end{align}
Solving~\eqref{eq:aggobj1} for each $i\in[n]$ and $j\in[N_i]$ would lead to the vectors $\bf (P^*,p^*)$,  which constitute an optimal solution for~\eqref{eq:aggobj}. We have the following result on the optimal decisions of the aggregator. 
\begin{lemma}\label{lem:aggopt}
	Under Assumption~\ref{assum:prosu}, consider an arbitrary prosumer~$j$ at location~$i$, with $z_1$ as in \eqref{eq:z}, and wholesale market price ~$\lambda^i$, When $z_1 < C^i_j$, the aggregator's optimal pricing decision is
	\begin{subequations}\label{eq:aggopt}
		\begin{align}
			{p^i_j}^* &= \lambda^i,\\
			{P^i_j}^* &= \lambda^i(C^i_j-z_1) + u^i_j(z_1) - u^i_j(C^i_j).
		\end{align}
	\end{subequations}
	When $z_1\ge C^i_j$, prosumer~$j$ will not sell DER supply to $\Acal$, i.e., $\Pi^i_j\left(P^i_j, p^i_j\right)=0$, for any $({P^i_j},{p^i_j})\in\Rset_+^2$.

\end{lemma}
Lemma~\ref{lem:aggopt} provides an optimal solution~$({P^i_j}^*, {p^i_j}^*)$ to~\eqref{eq:aggobj1}, and the collection ~$\bf (P^*, p^*)$ form an optimal solution to~\eqref{eq:aggobj}. We also note that the optimal solution may not be unique in general: $\Acal$ may deviate from~\eqref{eq:aggopt} by further decreasing~${p^i_j}^*$ and increasing~${P^i_j}^*$ to earn the same profit, while keeping the response of the prosumer~$j$ being unchanged. This lemma states that there is an optimal pricing scheme which sets the marginal price~$p^i_j$ to be the wholesale market price~$\lambda^i$, and all prosumers at the same location will be offered the same~$p^i_j = \lambda^i$. The participation fees~$P^i_j$, however, will be charged differently for prosumers with different utility functions. A keen reader would observe that in view of Lemmas \ref{lem:prosdirect} and \ref{lem:prosopt}, the optimal DER supply to the wholesale market is the same with and without the aggregator $\Acal$, thus making this aggregation model economically efficient with socially optimal market outcomes.

\subsection{Aggregator-Prosumers Interaction as a Stackelberg Game}
Let $\mathcal{G}(\pmb{\lambda})$ denote the Stackelberg game among the aggregator and the prosumers for a given vector of wholesale market prices  $\pmb{\lambda}$. Aggregator~$\mathcal{A}$ acts as a Stackelberg leader and sets the prices $\bf (P,p)$. The prosumer follows by responding with ${x^i_j}^*(P^i_j,p^i_j,\lambda^i)$. We now define the equilibrium of the game.
\begin{definition}
	$\left(\bf P^*,p^*, {x}^*(P^*,p^*)\right)$ constitutes a Stackelberg equilibrium of the game~$\mathcal{G}(\pmb{\lambda})$ if
	\begin{itemize}
		\item Prosumers: For any $x^i_j\in\left[0,C^i_j\right]$, we have  $$\pi_{j}^i\left({x^i_j}^*({P^i_j}, {p^i_j}),{P^i_j}, {p^i_j},\lambda^i\right)\ge \pi_{j}^i\left({x^i_j}, {P^i_j}^*, {p^i_j}^*,\lambda^i\right)$$ for all $i\in[n]$ and $j\in[N_i]$.
		\item Aggregator: For all ${\bf P \ge 0, p \ge 0}$, we have $$\hat{\Pi}({\bf P^*,p^*, x^*(P^*,p^*)}, \pmb{\lambda} ) \ge\hat{\Pi}(\bf{ P,p, x^*(P,p)}, \pmb{\lambda} ), $$ 
	\end{itemize}
\end{definition}
We then have the following Stackelberg equilibrium for the game~$\Gcal(\pmb{\lambda})$, which follows directly from the prosumer's optimal response (Lemma~\ref{lem:prosopt}) and the aggregator's optimal pricing (Lemma~\ref{lem:aggopt}).

\begin{proposition}\label{prop:equi}
	Under Assumption~\ref{assum:prosu}, the game $\mathcal{G}(\pmb{\lambda})$ admits a Stackelberg equilibrium that satisfies 
	\begin{subequations}
		\begin{align}
			&{P^i_j}^* = \left[\lambda^i(C^i_j-z^i_j) + u^i_j\left(z^i_j\right) - u^i_j(C^i_j)\right]^+\label{eq:equiP}\\
			&{p^i_j}^* = \lambda^i\label{eq:equip}\\
			&{x^i_j}^* ({P^i_j}^*,{p^i_j}^* )= \left[C^i_j - z^i_j\right]^+\label{eq:equix}
		\end{align}
	\end{subequations}
	for each prosumer~$j$ at each location~$i$, where $z^i_j$ satisfies $\frac{\partial u^i_j(z)}{\partial z}\Big|_{z=z^i_j} = \lambda^i$, and ${d^i_j}^*= \left[z^i_j- C^i_j\right]^+$.
\end{proposition}
We note that the game~$\Gcal(\pmb{\lambda})$ may admit other Stackelberg equilibria, but Proposition~\ref{prop:equi} provides the most economically efficient one. In case of non-uniqueness, this economically efficient equilibrium corresponds to the case where prosumers slightly prefer participation, i.e., if selling~$x^i_j$ amount of energy earns the prosumer~$j$ the same~$\pi^i_j$ as not selling, then she chooses to sell this~$x^i_j$. \footnote{Alternatively, one can impose a slight perturbation of ${P^i_j}^*$ to ${P^i_j}^*-\epsilon$, where $\epsilon>0$, to ensure maximum DER supply.} The above equilibrium is quite intuitive, $\Acal$ passes the location marginal price $\lambda^i$ obtained from wholesale market outcomes as is to prosumers, so $\Acal$ has no marginal profits from ${\bf p}^*$. Instead, $\Acal$ makes all of the profits from the upfront participation fees ${\bf P^*}$.  

\subsection{Generator's Problem}
For a given wholesale market price~$\lambda^i$, the conventional generators solve the same problem as described in Section~\ref{sec:benchgene}, and the result of Lemma~\ref{lem:geneopt} still applies under the current aggregation model.

\subsection{The Economic Dispatch Problem}
The system operator solves an optimization problem similar to that in the direct participation model as described in Section~\ref{sec:benchiso}. The network constraints~\eqref{eq:directcons} remain valid.
Besides, the ISO also considers participant-specific constraints:
\begin{equation}\label{eq:aggcons_p}
	{\bf 0\le x \le C, \quad 0\le d\le Z-C+x,\quad  \underline{y} \le y  \le \overline{y} } 	
\end{equation}
The objective of the system operator is to maximize the social welfare, which now includes the prosumer surplus~({\sf PS}), aggregator surplus~({\sf AS}), generator surplus~({\sf GS}), and merchandizing surplus ({\sf MS}):
\begin{subequations}
	\begin{align}
		{\sf PS} &:= \sum_{i\in[n]}\sum_{j\in [N_i]}\Bigg(u^i_j(d^i_j -x^i_j + {C}^i_j) - \lambda^i d^i_j \nonumber \\
		&\qquad\qquad\qquad\qquad +p^i_jx^i_j -  P^i_j\mathds{1}\left\{x^i_j>0\right\}\Bigg)\label{aggcs}\\
		{\sf AS} &:= \sum_{i\in[n]}\sum_{j\in[N_i]}\Bigg(P^i_j\mathds{1}\left\{x^i_j>0\right\}+\lambda^ix^i_j -p^i_jx^i_j  \Bigg)\\
		{\sf GS} &:= \sum_{i\in[n]}\sum_{j\in [G_i]}\left(\lambda^iy^i_j - c^i_j(y^i_j)\right), {\sf MS} := \sum_{i\in[n]}\lambda^ih^i.\label{aggms}
	\end{align}
\end{subequations}
The social welfare is the sum of the above four terms. By the supply-demand balance ${\bf h = D - Y - X}$, and after canceling terms, we write the social welfare as
\begin{align}
	\Wcal_A := \sum_{i\in[n]} \left(\sum_{j\in[N_i]}u^i_j(d^i_j +{C}^i_j-x^i_j) -  \sum_{j\in [G_i]}c^i_j(y^i_j)\right),
\end{align}
which is the same as~$\Wcal_B$ in~\eqref{eq:WB}.
The system operator's economic dispatch problem is then:
\begin{equation}\label{eq:soa}
	\begin{aligned}
		&\max \qquad \mathcal{W}_A ({\bf h, x, d, y})\\
		&  \text{subject to}\qquad \eqref{eq:directcons},\ \eqref{eq:aggcons_p}.
	\end{aligned}
\end{equation}
\begin{assumption}\label{assum:exist1}
	The system operator's economic dispatch problem~\eqref{eq:soa} is feasible.
\end{assumption}
As the system operator solves~\eqref{eq:soa}, the wholesale market prices~$\pmb{\lambda}$ are given by the optimal Lagrange multiplier of the constraint~\eqref{eq:demsup}. We then have the following proposition.
\begin{proposition}[Competitive Equilibrium]\label{prop:equi1}
	Under Assumptions~\ref{assum:prosu},~\ref{assum:gene},~\ref{assum:exist1}, there exists an optimal solution $\({\bf h^*, x^*, d^*, y^*}\)$ to~\eqref{eq:soa}. Let ${\pmb{\lambda}}$ denote the corresponding optimal Lagrange multipliers of constraints~\eqref{eq:demsup}. Then, the following statements are true: 
		\begin{itemize}
		\item $(\bf x^*,\bf {d}^*)$ are consistent with Lemma~\ref{lem:prosopt}, given $\bf(P^*,p^*)$ and $\pmb{\lambda}$. \vspace{0.05in}
		\item ${\bf(P^*,p^*)}$ are consistent with Lemma~\ref{lem:aggopt}, given ${\bf x^*}$ and $\pmb{\lambda}$.
		\item ${\bf y}^*$ satisfies~\eqref{eq:geneopt}, given $\pmb{\lambda}$.
		\vspace{0.05in}
	\end{itemize}
\end{proposition}

We now present the following theorem, which states that our proposed aggregation model achieves the same market efficiency as the benchmark direct participation model.
\begin{theorem}[Main Result]\label{thm:eff}
	Let $\Wcal_A^*$ be the optimal social welfare of~\eqref{eq:soa}, and let $\Wcal_B^*$ be the optimal social welfare of~\eqref{eq:so}. Then, we have that
	\begin{align}
		\Wcal_A^* = \Wcal_B^*.
	\end{align}
	Further, we have that the optimal ${\bf x^*,d^*,y^*}$ (from Proposition~\ref{prop:equi1}) solving~\eqref{eq:soa} are the same as those solving~\eqref{eq:so}.
\end{theorem}
In summary, under the proposed aggregation model, the aggregator procures energy from prosumers using two-part pricing, the aggregator would optimally pay the wholesale market price to the prosumers for each unit of energy procured, while the participation fee is differently charged to each prosumer as the additional consumer surplus when she sells those energy compared with not selling (which is dependent on her utility of consumption). Theorem~\ref{thm:eff}, together with Proposition~\ref{prop:equi1}, implies that under the aggregator's two-part differential pricing scheme, the prosumers' optimal buying and selling behavior is exactly the same as those in the direct participation model. As a result, the social welfare achieved under the aggregation model matches that of the direct participation model, i.e., there is no loss of market efficiency from the aggregation. 

\begin{remark}
	The significance of Theorem \ref{thm:eff} follows from the fact that via our proposed aggregation model,  DER aggregation through a profit-seeking aggregator $\Acal$ is equivalent to solving a socially-optimal economic dispatch model where $\Acal$ is absent. Hence, the potential efficiency loss due to presence of a monopolistic profit-seeking aggregator $\Acal$ is off-set by two-part pricing. This is not possible via one-part pricing, as demonstrated in \cite{CISS,TPS} and in Section IV, where efficiency loss arises from the profit-seeking behavior of $\Acal$.
\end{remark}

%
%

\section{ILLUSTRATIVE EXAMPLES}
\begin{figure*}
	\centering
	\includegraphics[width=.8\paperwidth]{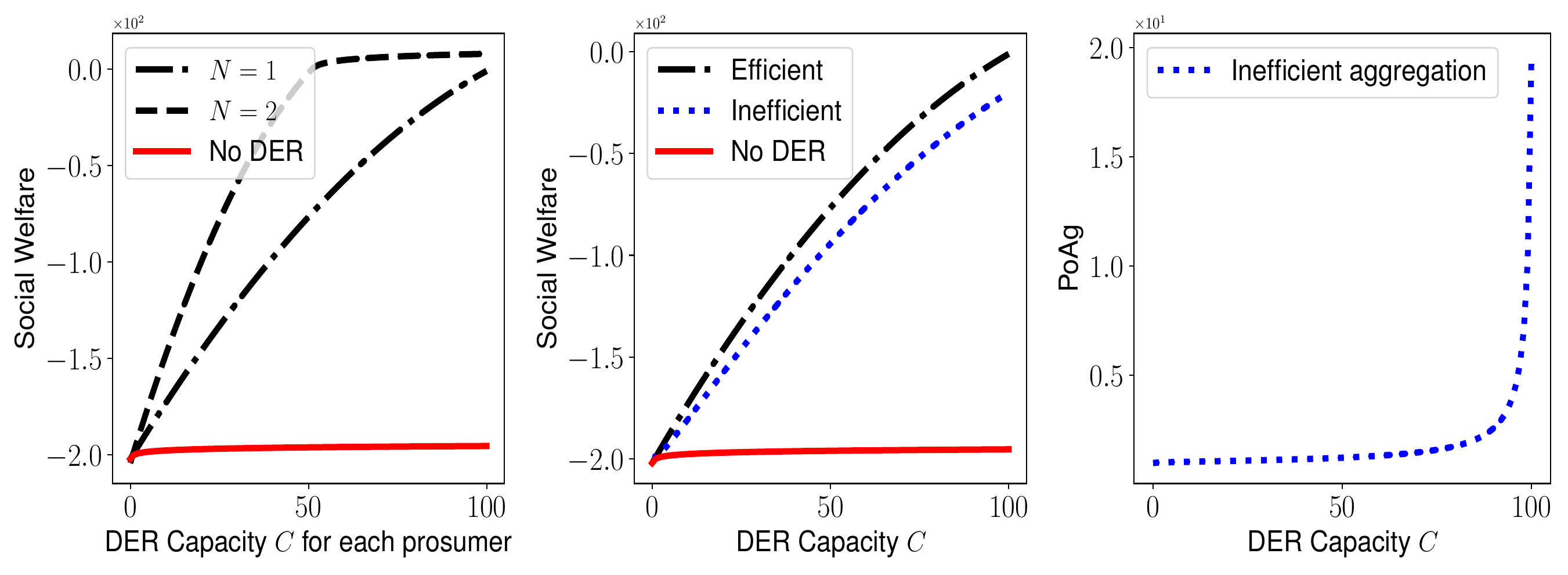}
	\caption{{\bf Left:} Efficient aggregation vs. no DER integration. Adding more prosumers attains a higher social welfare. {\bf Middle:} Comparison between the two extremes (efficient aggregation vs. no DER) and the one-part pricing model (inefficient). {\bf Right:} Quantification of the efficiency loss for the one-part pricing model, via the metric {\sf PoAg}.}
	\label{numerical}
\end{figure*}
In this section, we illustrate our results via stylized examples, with the goal of revealing some key insights. For ease of exposition, we consider a power system with 1 node, and hence, there are no network constraints. We remark that however, our theoretical results still hold for any power network. We also assume that there is a fixed demand $\bar{D}$, and thus~\eqref{eq:demsup} becomes
\begin{align}
	{ D + \bar{D} - Y - X = 0}.
\end{align}
We also consider one conventional power generator with cost $$c(y)=\alpha y^2 + \beta y.$$ 
  With one prosumer, consider the isoelastic utility function $u$, with risk-aversion parameter $\eta$  \cite{isoelastic}: 
\begin{align}\label{eq:isoelastic}
	u(z) = \begin{cases}
		\frac{z^{1-\eta}-1}{1-\eta} & \eta\ge 0, \eta \ne 1,\\
		\ln(z) & \eta = 1.
	\end{cases}
\end{align}

 When $\eta=0$, the prosumer is risk-neutral, for $\eta>0$, the prosumer is risk-averse, and increasing $\eta$ implies more risk-aversion.  For $C>p^{-1/\eta}$, from Proposition \ref{prop:equi}, the Stackelberg equilirium prices of the game $\Gcal({\lambda})$ are  $${p}^* = \lambda, \quad P^*= \left[\lambda\left(C-\lambda^{-1/\eta}\right) + u\left(\lambda^{-1/\eta}\right) - u(C)\right]^+,$$ 
and the equilibrium response is  $${x}^* ({P}^*,{p}^* )= C-\lambda^{-1/\eta}.$$ We note that by Lemma~\ref{lem:prosopt}, the case when $C\leq \lambda^{-1/\eta}$ corresponds to $x^* =0$  and $d^* =  \lambda^{-1/\eta} - C$. The ISO's problem for this single-prosumer case is then:
\begin{equation}\label{eq:exso}
	\begin{aligned}
		&\max \qquad \mathcal{W}_A ({ x, d, y}) := u(C-x+d) - c(y)\\
		&  \text{subject to}\qquad \bar{D}+d-x-y = 0,\quad 0\le y \le \bar{y},\\
		&\qquad\qquad\qquad 0\le x\le C,\quad 0\le d\le Z-C+x.
	\end{aligned}
\end{equation}
We compare the optimal value of~\eqref{eq:exso}, i.e., the optimal social welfare, with the following alternative models.

\noindent 1) \emph{No DER participation model}

	In this model, the prosumer is restricted to not selling her energy, i.e., $x=0$. 
	The ISO's problem is the same as~\eqref{eq:exso} with the additional constraint that $x=0$, i.e.,
	\begin{equation}\label{eq:exson}
		\begin{aligned}
			&\max \qquad \mathcal{W}_N ({d, y}) := u(C+d) - c(y)\\
			&  \text{subject to}\qquad \bar{D}+d-y = 0,\quad 0\le y \le \bar{y},\\
			&\qquad\qquad\qquad 0\le d\le Z-C.
		\end{aligned}
	\end{equation}

\noindent 2) \emph{One-part pricing model ($P=0$)}
	
	In this model, the prosumer may sell her energy to the aggregator~$\Acal$ at a fixed marginal price~$p$ set by~$\Acal$, with no participation fees. 
	In this model, the prosumer solves 
	\begin{equation}
		\begin{aligned}
			&\max_{x,d} \pi(x,d) = u(C+d-x) - \lambda d+px\\
			&\text{s.t. }\quad 0\le x\le C,\quad 0\le d\le Z-C+x.
		\end{aligned}
	\end{equation}
	With the isoelastic utility function~\eqref{eq:isoelastic}, an optimal response of the prosumer is $$d^* = \left[\lambda^{-1/\eta}-C\right]^+ \quad \text{and} \qquad x^* = \left[C-p^{-1/\eta}\right]^+.$$ Note that if $C\le \lambda^{-1/\eta}$, then $d^*\ge 0$ and $x^* =0$, which leads to zero profit of the aggregator for any $0\le p\le \lambda$, leading to no DER dispatch at the wholesale market level. If $C>\lambda^{-1/\eta}$, then the aggregator would choose a $p$ such that $C>p^{-1/\eta}$, and she solves	
		$$\max_{0\le p\le \lambda} \Pi=(\lambda - p)x^*(p) = (\lambda - p)\left(C-p^{-1/\eta}\right).$$
	  In this case, the Stackelberg equilibrium of the aggregator-prosumer game $\hat{\Gcal}({\lambda})$ (this is a different game from  $\Gcal({\lambda})$, but similarly defined) is given by $(x^*(p^*), p^*)$ that satisfy
	 $$(1-\eta)p^* + \eta C{p^*}^{1+1/\eta}=\lambda, \quad x^*(p^*) = C-p^{-1/\eta}.$$ 
	 If we let $\eta=1$ (logarithmic utility), we then have 
	 $$p^* (\lambda)= \sqrt{\frac{\lambda}{C}}, \quad x^*(p^*) = C-p^{-1}.$$ 	 
	 Now, the next question is, how can we define the corresponding social welfare. Note that here, the Stackelberg equilibrium $(x^*(p^*), p^*)$  does not yield an efficient wholesale market outcome, so instead of using prosumer's utility in ISO's economic dispatch problem, we need to construct an induced function. First, we note that $$x^*(p^*(\lambda) )= C-\sqrt{\frac{C}{\lambda}},$$ and hence, the corresponding inverse supply function for the prosumer is $$p_A(x) = \frac{C}{(C-x)^2}.$$

	
	Now, the system operator solves the economic dispatch problem given by
	\begin{equation}\label{eq:exsoo}
		\begin{aligned}
			&\max\qquad \Wcal_O(x,y):= -c(y) - \int_{0}^{x} p_A(\hat{x})d\hat{x}\\
			&  \text{subject to}\qquad \bar{D}-x-y = 0, \quad 0 \le y\le \bar{y},\\
			&\qquad\qquad\qquad 0\le x\le C.
		\end{aligned}
	\end{equation}
Since there is no elastic demand $d$, the above problem is equivalent to the following cost minimization problem
	\begin{equation}\label{eq:exsoo}
		\begin{aligned}
			&\min\qquad \Ccal_O(x,y):= c(y) + \frac{x}{C-x}\\
			&  \text{subject to}\qquad \bar{D}-x-y = 0, \quad 0 \le y\le \bar{y},\\
			&\qquad\qquad\qquad 0\le x\le C.
		\end{aligned}
	\end{equation}
Anaolgously, we can define the cost $\Ccal_A:=-\Wcal_A$ for the efficient dispatch model. To capture the efficiency loss for this model, we adopt the Price of Aggregation ({\sf PoAg}) metric introduced in \cite{CISS,TPS}, which is defined as $$ \text{\sf PoAg}:= \frac{\Ccal^*_O}{\Ccal^*_A}.$$ One can look at {\sf PoAg} as similar to the Price-of-Anarchy in noncooperative games, but {\sf PoAg} naturally captures the hierarchical nature of electricity markets. By definition, $\text{\sf PoAg}\geq1$, and when it is equal to $1$, it means that the aggregation model yields a socially-optimal market outcome. Note that {\sf PoAg} can be analogously defined for social welfare maximization, to compare any inefficient model to  $\Wcal_A^*$.

Figure \ref{numerical} illustrates the market outcomes for the above models. In our simulations, we vary the capacity $C$ from 0 to $\bar{D}=100$. One can think of this as a proxy for renewable energy integration, $C=0$ implies \%0 renewable integration, and $C=\bar{D}$ implies \%100 of the total inflexible load might be met exclusively with DER capacity. We also assume the following parameters $$\eta=1,\  \alpha=0.01, \  \beta=1, \ Z=1000,\  \bar{y}=1000.$$ The above parameters are picked such that it is more expensive to use the conventional generator than DERs, and conventional generators can fully meet the total system's demand. DER aggregation is expected to improve the social welfare, and this is demonstrated in Figure \ref{numerical} ({\bf Left}). Since our model is efficient, the improvements shown are the maximum possible ones. We then increase the number of prosumers to two, having the same capacity and $\eta$, and observe that the efficient welfare further improves. This is natural: the more DER capacities available, the more cheaper resources are available to the ISO. Next, we fix the number of prosumers to one, and study other aspects. The one-part pricing model discussed above is also expected to improve the social welfare, compared to no DER participation, but remains inefficient, as Figure \ref{numerical} ({\bf Middle}) shows. We quantify the efficiency loss via the {\sf PoAg} metric and plot it in Figure \ref{numerical} ({\bf Right}). Here, {\sf PoAg} increases nonlinearly as $C$ increases. 

\section{CONCLUSIONS AND FUTURE WORKS}

To include DER capacities in the wholesale electricity market, we proposed an efficient DER aggregation model. The aggregator announces a two-part price offer to each prosumer (DER owner): a participation fee and a per-unit price paid to the prosumer for the energy procured. This interaction is modeled as a Stackelberg game, which admits an efficient equilibrium. Under this equilibrium, the per-unit price is equal to the wholesale market price, and prosumers choose to sell the same amount of energy as if they were participating directly in the wholesale market (ideal case). As a result, the social welfare achieved under this aggregation model is the same as the social welfare achieved when prosumers can sell directly, i.e., there is no efficiency loss even with a monopolistic aggregator. The two-part pricing mechanism is essential in achieving this efficiency result. As illustrated via an example, there may be efficiency loss if no participation fees can be charged. Under the current model, the aggregator charges different participation fees to the prosumers. There may be some debates on whether differential pricing on the participation fee should be allowed. If the aggregator has to choose a single participation fee for all prosumers at the same location, we would no longer have the efficiency result, and characterizing the efficiency loss in this aggregation process would be a promising direction to work on. Another direction for future work is to consider stochastic DER capacities under two-part pricing, and see how the stochastic nature of the DER supply~\cite{sunar2019strategic} would affect the aggregator-prosumer game and the overall market efficiency.

\bibliographystyle{IEEEtran}
\bibliography{references}

\begin{thebibliography}{10}
\providecommand{\url}[1]{#1}
\csname url@samestyle\endcsname
\providecommand{\newblock}{\relax}
\providecommand{\bibinfo}[2]{#2}
\providecommand{\BIBentrySTDinterwordspacing}{\spaceskip=0pt\relax}
\providecommand{\BIBentryALTinterwordstretchfactor}{4}
\providecommand{\BIBentryALTinterwordspacing}{\spaceskip=\fontdimen2\font plus
\BIBentryALTinterwordstretchfactor\fontdimen3\font minus
  \fontdimen4\font\relax}
\providecommand{\BIBforeignlanguage}[2]{{%
\expandafter\ifx\csname l@#1\endcsname\relax
\typeout{** WARNING: IEEEtran.bst: No hyphenation pattern has been}%
\typeout{** loaded for the language `#1'. Using the pattern for}%
\typeout{** the default language instead.}%
\else
\language=\csname l@#1\endcsname
\fi
#2}}
\providecommand{\BIBdecl}{\relax}
\BIBdecl

\bibitem{ritzenhofen2016structural}
I.~Ritzenhofen, J.~R. Birge, and S.~Spinler, ``The structural impact of
  renewable portfolio standards and feed-in tariffs on electricity markets,''
  \emph{European Journal of Operational Research}, vol. 255, no.~1, pp.
  224--242, 2016.

\bibitem{derdefinition}
``Distributed {E}nergy {R}esouces: {C}onnection {M}odeling and {R}eliability
  {C}onsiderations,'' North {A}merican {E}lectric {R}eliability {C}orporation,
  Tech. Rep., 2017.

\bibitem{eugene}
D.~Bertsimas, E.~Litvinov, X.~A. Sun, J.~Zhao, and T.~Zheng, ``Adaptive robust
  optimization for the security constrained unit commitment problem,''
  \emph{IEEE Transactions on Power Systems}, vol.~28, no.~1, pp. 52--63, 2013.

\bibitem{eugene2}
D.~Gan and E.~Litvinov, ``Energy and reserve market designs with explicit
  consideration to lost opportunity costs,'' \emph{IEEE Transactions on Power
  Systems}, vol.~18, no.~1, pp. 53--59, 2003.

\bibitem{eugene3}
E.~Litvinov, T.~Zheng, G.~Rosenwald, and P.~Shamsollahi, ``Marginal loss
  modeling in {LMP} calculation,'' \emph{IEEE Transactions on Power Systems},
  vol.~19, no.~2, pp. 880--888, 2004.

\bibitem{EnergyReserve}
T.~Zheng and E.~Litvinov, ``Ex post pricing in the co-optimized energy and
  reserve market,'' \emph{IEEE Trans. Power Systems}, vol.~21, no.~4, pp. 1528
  -- 1538, 2006.

\bibitem{feng}
F.~Zhao, T.~Zheng, and E.~Litvinov, ``Constructing demand curves in forward
  capacity market,'' \emph{IEEE Trans. Power Systems}, vol.~33, no.~1, pp.
  525--535, 2018.

\bibitem{Lian}
J.~{Lian}, D.~{Wu}, K.~{Kalsi}, and H.~{Chen}, ``Theoretical framework for
  integrating distributed energy resources into distribution systems,'' in
  \emph{2017 IEEE Power Energy Society General Meeting}, July 2017, pp. 1--5.

\bibitem{Ntakou}
E.~{Ntakou} and M.~{Caramanis}, ``Price discovery in dynamic power markets with
  low-voltage distribution-network participants,'' in \emph{2014 IEEE PES T\&D
  Conference and Exposition}, April 2014, pp. 1--5.

\bibitem{Manshadi}
S.~D. {Manshadi} and M.~E. {Khodayar}, ``A hierarchical electricity market
  structure for the smart grid paradigm,'' \emph{IEEE Transactions on Smart
  Grid}, vol.~7, no.~4, pp. 1866--1875, July 2016.

\bibitem{Sotkiewicz}
P.~M. {Sotkiewicz} and J.~M. {Vignolo}, ``Nodal pricing for distribution
  networks: efficient pricing for efficiency enhancing {DG},'' \emph{IEEE
  Transactions on Power Systems}, vol.~21, no.~2, pp. 1013--1014, May 2006.

\bibitem{Terra}
J.~{Terra}, R.~{Ferreira}, C.~{Borges}, and M.~{Carvalho}, ``Using
  distribution-level locational marginal pricing to value distributed
  generation: Impacts on revenues captured by generation agents,'' in
  \emph{2017 IEEE PES Innovative Smart Grid Technologies Conference - Latin
  America (ISGT Latin America)}, Sep. 2017, pp. 1--6.

\bibitem{Huang}
S.~{Huang}, Q.~{Wu}, S.~S. {Oren}, R.~{Li}, and Z.~{Liu}, ``Distribution
  locational marginal pricing through quadratic programming for congestion
  management in distribution networks,'' \emph{IEEE Transactions on Power
  Systems}, vol.~30, no.~4, pp. 2170--2178, July 2015.

\bibitem{Moret}
F.~{Moret} and P.~{Pinson}, ``Energy collectives: a community and fairness
  based approach to future electricity markets,'' \emph{IEEE Transactions on
  Power Systems}, pp. 1--1, 2018.

\bibitem{Rahimi}
F.~A. Rahimi and A.~Ipakchi, ``Transactive energy techniques: Closing the gap
  between wholesale and retail markets,'' \emph{The Electricity Journal},
  vol.~25, no.~8, pp. 29 -- 35, 2012.

\bibitem{CAISODER}
J.~Gundlach and R.~Webb, ``Distributed energy resource participation in
  wholesale markets: Lessons from the {C}alifornia {ISO},'' \emph{Energy Law
  Journal}, vol.~39, no.~1, pp. 47--77, 2018.

\bibitem{DER_NYISO}
M.~Lavillotti, ``{DER} market design: Aggregations,'' New York ISO, Tech. Rep.,
  2018.

\bibitem{FERC2222}
FERC, ``Participation of distributed energy resource aggregations in markets
  operated by regional transmission organizations and independent system
  operators,'' 2020, {O}rder No. 2222.

\bibitem{CISS}
K.~Alshehri, M.~Ndrio, S.~Bose, and T.~Ba\c{s}ar, ``The impact of aggregating
  distributed energy resources on electricity market efficiency,'' \emph{Proc.
  53rd Annual Conference on Information Systems and Sciences, Johns Hopkins
  University, Baltimore, Maryland}, March 20-22, 2019.

\bibitem{TPS}
------, ``Quantifying market efficiency impacts of aggregated distributed
  energy resources,'' \emph{IEEE Transactions on Power Systems}, vol.~35,
  no.~5, 2020.

\bibitem{CAISO}
``Distributed energy resource provider participation guide with checklist,''
  California ISO, Tech. Rep., 2016.

\bibitem{basar}
T.~Ba\c{s}ar and G.~J. Olsder, \emph{Dynamic Noncooperative Game Theory}.\hskip
  1em plus 0.5em minus 0.4em\relax SIAM, 1999.

\bibitem{birge2017inverse}
J.~R. Birge, A.~Horta{\c{c}}su, and J.~M. Pavlin, ``Inverse optimization for
  the recovery of market structure from market outcomes: An application to the
  {MISO} electricity market,'' \emph{Operations Research}, vol.~65, no.~4, pp.
  837--855, 2017.

\bibitem{isoelastic}
L.~Ljungqvist and T.~J. Sargent, \emph{Recursive Macroeconomic Theory}.\hskip
  1em plus 0.5em minus 0.4em\relax MIT Press, 2000.

\bibitem{sunar2019strategic}
N.~Sunar and J.~R. Birge, ``Strategic commitment to a production schedule with
  uncertain supply and demand: Renewable energy in day-ahead electricity
  markets,'' \emph{Management Science}, vol.~65, no.~2, pp. 714--734, 2019.

\end{thebibliography}

\section*{APPENDIX}

\begin{proof}[Proof of Proposition~\ref{prop:compeq}]
	First note that~\eqref{eq:so} can be equivalently written as:
	\begin{equation}\label{eq:soeq}
		\begin{aligned}
			&\max \qquad \mathcal{W}_B ({\bf x-d, y})\\
			&  \text{subject to}\ \ {\bf 1}^\top \left({\bf X-D+Y}\right)=\bf  0, \ {\bf -B(X-D+Y)}\le {\bf f},\\
			& \qquad\qquad \ \ {\bf C-Z\le x-d \le C},  \  {\bf \underline{y} \le y  \le \overline{y}. } 
		\end{aligned}
	\end{equation}
	By Assumption~\ref{assum:prosu} and Assumption~\ref{assum:gene}, each $u^i_j$ is strictly concave, and each $c^i_j$ is strictly convex. Thus, $\Wcal_B$ given by~\eqref{eq:WB} is strictly concave. It can be seen from~\eqref{eq:soeq} that the feasible set is compact, and is nonempty by Assumption~\ref{assum:exist}. Therefore, there exists a unique optimal solution $\({\bf (x-d)^*, y^*}\)$ to~\eqref{eq:soeq}. By letting ${\bf h^*} = {\bf -Y^* - (X-D)^*}$, we have a unique optimal solution $\({\bf h^*, (x-d)^*, y^*}\)$ to~\eqref{eq:so}.
	
	We write the Lagrangian of~\eqref{eq:so} as
		\begin{align}
			\Lcal =& \sum_{i\in[n]} \Bigg(\sum_{j\in[N_i]}u^i_j\left({C}^i_j-(x^i_j-d^i_j)\right) -  \sum_{j\in [G_i]}c^i_j(y^i_j) \nonumber\\
			& + \lambda^i\left(h^i+X^i-D^i+Y^i\right) + \sum_{j\in[N_i]}\Big(\overline{\mu}^i_j\left(C^i_j-(x^i_j-d^i_j)\right)\nonumber\\
			& +\underline{\mu}^i_j\left(x^i_j-d^i_j-C^i_j + Z\right)\Big) + \sum_{j\in[G_i]}\Big(\overline{\nu}^i_j(\overline{y}^i_j-y^i_j)\nonumber\\
			& + \underline{\nu}^i_j(y^i_j-\underline{y}^i_j)\Big)\Bigg) + \text{extra terms},
		\end{align}
	where $\lambda^i, \overline{\mu}^i_j,\underline{\mu}^i_j, \overline{\nu}^i_j, \underline{\nu}^i_j$ are the Lagrange multipliers corresponding to constraints~\eqref{eq:demsup} and~\eqref{eq:directcons_p}, and the ``extra terms" correspond to constraint~\eqref{eq:netcons}. Under Assumptions~\ref{assum:prosu} and~\ref{assum:gene}, it should be clear that the optimal solution $(x^i_j-d^i_j)^*\neq C^i_j$, $(x^i_j-d^i_j)^*\neq C^i_j-Z$, and ${y^i_j}^*\ne \overline{y}^i_j$, ${y^i_j}^*\ne \underline{y}^i_j$. Then, from the KKT optimality conditions, we have that $\overline{\mu}^i_j = \underline{\mu}^i_j = \overline{\nu}^i_j = \underline{\nu}^i_j = 0$. Further, we have that 
	\begin{subequations}
		\begin{align}
			\nabla_{(x^i_j-d^i_j)}\Lcal &= \frac{\partial u^i_j}{\partial (x^i_j-d^i_j)}\Big|_{(x^i_j-d^i_j)^*} + \lambda^i = 0,\label{eq:grad1}\\
			\nabla_{y^i_j}\Lcal &= -\frac{\partial c^i_j}{\partial y^i_j}\Big|_{{y^i_j}^*} + \lambda^i = 0,\label{eq:grad2}
		\end{align}
	\end{subequations}
	where~\eqref{eq:grad1} is equivalent to~\eqref{eq:prosopt} and~\eqref{eq:grad2} is equivalent to~\eqref{eq:geneopt}. 
\end{proof}

\begin{proof}[Proof of Lemma~\ref{lem:prosopt}]
	Consider prosumer~$j$ at location~$i$. By Assumption~\ref{assum:prosu}, $\frac{\partial u^i_j(z)}{\partial z}$ is continuous and ranges from~$\infty$ to~$0$. By Intermediate Value Theorem, for any given~$\lambda^i$ and~$p^i_j$, there exist a $(z_1,z_2)$ such that~\eqref{eq:z} holds. By strict concavity of~$u^i_j$, $\frac{\partial u^i_j(z)}{\partial z}$ is strictly decreasing, so the $(z_1,z_2)$ is unique.
	
	We next find the optimal solution to~\eqref{eq:aggpros}. We look at the two cases when $x^i_j >0$ and $x^i_j = 0$ separately. Recall that $p^i_j\le \lambda^i$. Then for any $x^i_j > 0$ and $d^i_j \ge x^i_j$, we have that
	\begin{align*}
		\pi^i_j\left(x^i_j,d^i_j\right)\le \pi^i_j\left(0, d^i_j-x^i_j\right).
	\end{align*}
	For any $x^i_j > 0$ and $0<d^i_j< x^i_j$, we have that 
	\begin{align*}
		\pi^i_j\left(x^i_j,d^i_j\right)\le \pi^i_j\left(x^i_j-d^i_j, 0\right).
	\end{align*}
	Therefore, we may without loss of generality restrict to the case when $x^i_j$ and $d^i_j$ are not both strictly positive. We can rewrite~\eqref{eq:aggpros} as
	\begin{align}\label{eq:aggpros1}
		&\pi_{j}^i(x^i_j, d^i_j)=
		\begin{cases}
			p^i_jx^i_j - P^i_j + u^i_j\left({C}^i_j-x^i_j\right), \text{ if } x^i_j>0,\\
			u^i_j\left(d^i_j + {C}^i_j\right) - \lambda^id^i_j,
			\hfill \text{ if }x^i_j=0.
		\end{cases} 
	\end{align}
\begin{itemize}
	\item If $C^i_j\le z_2$, then $u^i_j\left(C^i_j-x^i_j\right) + p^i_jx^i_j\le u^i_j\left(C^i_j\right)$ for any $x^i_j > 0$. Thus we have the optimal ${x^i_j}^* = 0$. If we further have that $C^i_j < z_1$, then the first order condition of $u^i_j\left(d^i_j+C^i_j\right) -\lambda^id^i_j$ leads to ${d^i_j}^* = z_1-C^i_j$. If we instead have $C^i_j\ge z_1$, then $u^i_j\left(d^i_j+C^i_j\right) - \lambda^id^i_j \le u^i_j\left(C^i_j\right)$ for any $d^i_j > 0$.
	\item If $C^i_j > z_2$, then $u^i_j\left(d^i_j+C^i_j\right)-\lambda^i d^i_j < u^i_j\left(C^i_j\right)$ for any $d^i_j > 0$. We thus have ${d^i_j}^* = 0$. When $x^i_j > 0$, the first order condition of $p^i_jx^i_j - P^i_j + u^i_j\left({C}^i_j-x^i_j\right)$ leads to ${x^i_j} = C^i_j-z_2$. Moreover, it is only optimal to have ${x^i_j}^* = C^i_j-z_2>0$ if $p^i_jx^i_j - P^i_j + u^i_j\left({C}^i_j-x^i_j\right) \ge u^i_j\left(C^i_j\right)$, or equivalently, $P^i_j\le p^i_j\left(C^i_j-z_2\right)+ u^i_j\left(z_2\right) - u^i_j\left(C^i_j\right)$. Therefore, we have that ${x^i_j}^*= \left(C^i_j-z_2\right)\cdot \mathds{1}\left\{\Xcal\right\}$, where $\Xcal=\{P^i_j\le p^i_j\left(C^i_j-z_2\right)+ u^i_j\left(z_2\right) - u^i_j\left(C^i_j\right)\}$.
\end{itemize}
\end{proof}

\begin{proof}[Proof of Lemma~\ref{lem:aggopt}]
	Consider prosumer~$j$ at location~$i$. Let $(z_1,z_2)$ be as in~\eqref{eq:z}. Then $z_1\le z_2$ since $p^i_j\le \lambda^i$. First note that when $C^i_j\le z_1\le z_2$, ${x^i_j}^* = 0$ by Lemma~\ref{lem:prosopt}, and $\Pi^i_j\left(P^i_j,p^i_j\right) = 0$ for any $(P^i_j,p^i_j)\in\mathbb{R}^2_+$, i.e., the aggregator~$\Acal$ earns zero profit from the prosumer regardless of her pricing decisions. 
	
	For those prosumers with $C^i_j > z_1$, we have that ${d^i_j}^* = 0$ from Lemma~\ref{lem:prosopt}.  Since the aggregator earns zero profit if the prosumer does not sell her energy, the aggregator would choose a~$p^i_j$ such that $C^i_j> z_2$ (because otherwise ${x^i_j}^*=0$) and thus ${x^i_j}^* = C^i_j-z_2$. We can add in the aggregator's problem the constraint that the prosumers would benefit from selling:
	\begin{equation}\label{eq:agg1}
		\begin{aligned}
			&\max_{P^i_j,p^i_j\ge 0} P^i_j + \left(\lambda^i-p^i_j\right){x^i_j}^*\\
			&\text{s.t. } p^i_j{x^i_j}^*+ u^i_j\left(C^i_j-{x^i_j}^*\right) - P^i_j\ge u^i_j\left(C^i_j\right),
		\end{aligned}
	\end{equation}
	where ${x^i_j}^*$ is the optimal response of the prosumer, as a function of $(P^i_j,p^i_j)$. We observe from~\eqref{eq:agg1} that, the optimal $P^i_j$ should satisfy $P^i_j =p^i_j{x^i_j}^*+ u^i_j\left(C^i_j-{x^i_j}^*\right) - u^i_j\left(C^i_j\right)$ in the maximization problem. Thus,~\eqref{eq:agg1} can be rewritten as
	\begin{align}
		&\max_{p^i_j\ge 0}p^i_j{x^i_j}^*+ u^i_j\left(C^i_j-{x^i_j}^*\right) - u^i_j\left(C^i_j\right) + \left(\lambda^i-p^i_j\right){x^i_j}^* \nonumber\\
		=&\max_{z_2\ge 0} u^i_j(z_2) - u^i_j\left(C^i_j\right)+\lambda^i\left(C^i_j-z_2\right),\label{eq:agg2}
	\end{align}
	and the first order condition of~\eqref{eq:agg2} leads to 
	\begin{align}
		\frac{\partial u^i_j(z_2)}{\partial z_2}\Big|_{z_2=z_2^*} = \lambda^i.
	\end{align}
	By~\eqref{eq:z}, we also have that $\frac{\partial u^i_j(z)}{\partial z}\Big|_{z=z_2^*} = {p^i_j}^*$. We thus conclude that ${p^i_j}^* = \lambda^i$. This further leads to $z_2^* = z_1$, and thus 
	\begin{align*}
		{P^i_j}^* &={p^i_j}^*\left(C^i_j-z_2^*\right)+ u^i_j\left(z_2^*\right) - u^i_j\left(C^i_j\right)\\
		&= \lambda^i\left(C^i_j-z_1\right) + u^i_j\left(C^i_j-z_1\right) - u^i_j\left(C^i_j\right).
	\end{align*}
\end{proof}


\begin{proof}[Proof of Proposition~\ref{prop:equi1}]
	In this proof, we show that the optimal solution to~\eqref{eq:so} can be used to construct an optimal solution to~\eqref{eq:soa}. 
	
	First note that the objectives of both problems have the same expression, i.e., $\Wcal_A({\bf h,x,d,y}) = \Wcal_{B}({\bf h,x-d,y}) = \sum_{i\in[n]} \left(\sum_{j\in[N_i]}u^i_j(d^i_j +{C}^i_j-x^i_j) -  \sum_{j\in [G_i]}c^i_j(y^i_j)\right).$ While we took $(\bf x-d)$ as a single vector of decision variables in solving~\eqref{eq:so} and obtained a unique optimal solution $\({\bf h^*, (x-d)^*, y^*}\)$ in Proposition~\ref{prop:compeq}, we can equivalently consider ${\bf x}$ and $\bf d$ as two separate vectors of decision variables. Consider the constraints~\eqref{eq:directcons_p} and constraints~\eqref{eq:aggcons_p}. For any ${\bf x,d,y}$ satisfying constraints~\eqref{eq:aggcons_p}, we have that ${\bf x-d\le C}$ and ${\bf x-d \ge C-Z}$, which implies constraints~\eqref{eq:directcons_p}. With the other constraints being the same, the feasible region of~\eqref{eq:soa} is a subset of the feasible region of~\eqref{eq:so}, which implies that $\Wcal_B^* \ge \Wcal_A^*$.
	
	Let $\({\bf h^*, (x-d)^*, y^*}\)$ be the optimal solution to~\eqref{eq:so}. Let $\bf x^* = \left[(x-d)^*\right]^+$ and $\bf d^* = \left[-(x-d)^*\right]^+$, then $\bf x^*-d^* = (x-d)^*$. We show that $\({\bf h^*, x^*, d^*, y^*}\)$ is feasible to~\eqref{eq:soa}. It suffices to show that $\bf 0\le x^*\le C$ and $\bf 0\le d^*\le Z-C+x^*$. By definition, $\bf x^*\ge 0$ and $\bf d^* \ge 0$. If $\bf (x-d)^* < 0$, then $\bf x^* = 0$ and $\bf d^* = -(x-d)^*\le Z-C = Z-C+x^*$. If $\bf (x-d)^* \ge 0$, then $\bf x^* = (x-d)^* \le C$ and $\bf d^* = 0$. In either case, we have that $\bf x^*,d^*$ are feasible to~\eqref{eq:aggcons_p}, and thus $\({\bf h^*, x^*, d^*, y^*}\)$ is feasible to~\eqref{eq:soa}. Therefore, this set of $\({\bf h^*, x^*, d^*, y^*}\)$ is an optimal solution to~\eqref{eq:soa}, and thus $\Wcal_A^* = \Wcal_B^*$. Further, the corresponding optimal Lagrange multipliers of constraints~\eqref{eq:demsup} are the same in problem~\eqref{eq:so} and problem~\eqref{eq:soa}, i.e., the wholesale market prices~${\pmb{\lambda}}$ are the same under both models.
	
	From Proposition~\ref{prop:compeq}, we know that ${y^i_j}^*$  and $\lambda_i$ satisfy~\eqref{eq:geneopt},  $\forall  i\in[n], \forall j\in[G_i]$. Moreover, $\left(x^i_j-d^i_j\right)^*$ and $\lambda^i$ satisfy~\eqref{eq:prosopt}, $ \forall i\in[n], \forall j \in[N_i]$. Consider an arbitrary prosumer~$j$ at an arbitrary location~$i$. With $p^i_j = \lambda^i$, $(z_1,z_2)$ as defined in~\eqref{eq:z}, and ${P^i_j} = \left[\lambda^i(C^i_j-z_1) + u^i_j(z_1) - u^i_j(C^i_j)\right]^+$, we have that $z_1=z_2 = C^i_j-(x^i_j-d^i_j)^*$. If $C^i_j \le z_2 = C^i_j-(x^i_j-d^i_j)^*$, then $(x^i_j-d^i_j)^* <0$, which implies that ${d^i_j}^* = \left[-(x^i_j-d^i_j)^*\right]^+ = \left[z_1-C^i_j\right]^+$ and ${x^i_j}^* =\left[(x^i_j-d^i_j)^*\right]^+= 0$. If $C^i_j > z_2 = C^i_j-(x^i_j-d^i_j)^*$, then $(x^i_j-d^i_j)^* >0$, which implies that ${d^i_j}^* = \left[-(x^i_j-d^i_j)^*\right]^+ = 0$ and ${x^i_j}^* = \left[(x^i_j-d^i_j)^*\right]^+ = C^i_j-z_2$. In either case, the optimal response ${x^i_j}^*$ and ${d^i_j}^*$ matches those described in Lemma~\ref{lem:prosopt}. Further, by Lemma~\ref{lem:aggopt}, the choice of $p^i_j = \lambda^i$ and ${P^i_j} = \left[\lambda^i(C^i_j-z_1) + u^i_j(z_1) - u^i_j(C^i_j)\right]^+$ are optimal for the aggregator to maximize her profit from the prosumer~$j$. We have thus verified all statements listed in Proposition~\ref{prop:equi1}.\end{proof}

\begin{proof}[Proof of Theorem~\ref{thm:eff}]
	Immediately follows from the proof of Proposition~\ref{prop:equi1}.
\end{proof}

\end{document}